\documentclass[12pt, letterpaper]{amsart}
\usepackage[utf8]{inputenc}
\usepackage{amssymb}
\usepackage{amsmath}
\usepackage{cite}
\usepackage{graphicx}
\usepackage{float}
\usepackage{hyperref}
\usepackage{cleveref}
\usepackage[margin=1in]{geometry}
\usepackage{mathtools}
\usepackage{mathrsfs}
\usepackage{xcolor}

\newcommand{\bz}{\bar{z}}
\newcommand{\tu}{\tilde{u}}
\newcommand{\CI}{\mathcal{C}^{\infty}}

\newcommand{\RR}{\mathbf{R}}

\newcommand{\bP}{\bar{P}}

\newcommand{\Lap}{\Delta}

\newcommand{\RNum}[1]{\uppercase\expandafter{\romannumeral #1\relax}}

\newcommand{\df}{\coloneqq}
\newcommand{\la}{\langle}
\newcommand{\ra}{\rangle}
\newcommand{\RRR}{\RR^2\backslash\{0\}}
\newcommand{\pa}{P_{\alpha}}

\newcommand{\dom}{\mathcal{D}}

\newtheorem{theorem}{Theorem}[section]
\newtheorem*{theorem*}{Theorem}

\newtheorem{lemma}[theorem]{Lemma}
\newtheorem{definition}{Definition}[section]
\newtheorem{proposition}[theorem]{Proposition}
\newtheorem{remark}{Remark}[section]

\title{Diffraction of the Aharonov--Bohm Hamiltonian}
\author{Mengxuan Yang}
\date{\today}
\address{Department of Mathematics, Northwestern University}
\email{mxyang@math.northwestern.edu}

\begin{document}

\maketitle

\begin{abstract}
    In this paper, we compute the diffractive wave propagator of the Aharonov--Bohm effect \cite{aharonov1959significance} on $\RR^2$ with a single solenoid using a technique of moving solenoid location. In addition, we compute the corresponding diffraction coefficient which is the principal symbol of the diffractive propagator. This paper proves the propagation of singularities of the Aharonov--Bohm Hamiltonian.
\end{abstract}
\section{Introduction}
In this paper we study the diffraction behavior of the electromagnetic Hamiltonian on $\RR^2$:  
\begin{equation*}
    \pa=\left(\frac{1}{i}\nabla-\Vec{A}\right)^2,
\end{equation*}
where $\Vec{A}=-\alpha\left(-\frac{y}{x^2+y^2}, \frac{x}{x^2+y^2} \right)^T$ is an electromagnetic vector potential and $\alpha$ depends on the magnetic field strength at $r=0$. This operator corresponds to the model of one infinitely thin and long solenoid placed at the $z$-axis in $\RR^3$ with the aforementioned vector potential $\Vec{A}$ and the magnetic flux only at $z$-axis; the translation invariance along $z$-axis therefore reduces the model to the Hamiltonian $\pa$ on $\RR^2$.

The presence of such a vector potential generates a $\delta$-type magnetic field
$$\Vec{B}=\nabla\times\Vec{A}=\Lap\log(|\mathbf{x}|)=-2\pi\alpha\delta,$$
which was first studied by Aharonov and Bohm in \cite{aharonov1959significance} to show the significance of electromagnetic vector potentials in quantum mechanics. Away from the solenoid, although there is no magnetic field, quantum particles still experience a phase shift while passing two sides of the solenoid, which is not observable from the classical mechanical viewpoint. The phase difference generates an interference pattern which is called the Aharonov--Bohm effect. 

The diffraction refers to the effect that when a propagating wave/quantum particle encounters a corner of an obstacle or a slit, its wave front bends around the corner of the obstacle and propagates into the geometrical shadow region. When studying the wave equation on a singular domain, the singularities of the wave equation likewise split into two types after they encounter the singularity of the domain. One propagates along the natural geometric extension of the incoming ray, which is called the geometric wave front, while other singularities emerge at the cone point and start propagating along all outgoing directions as a spherical wave, which is called the diffractive wave. The singular vector potential $\Vec{A}$, similar to a singular domain, also generates a diffractive wave which is the central object of our study. For a detailed discussion of the diffractive wave and the geometric wave along with the corresponding propagators, see Section \ref{background}. Note that the Hamiltonian with the singular vector potential is closely related to the Laplacian on cones, for which the diffraction has been investigated thoroughly in \cite{cheeger1982diffraction}\cite{melrose2004propagation}.  

We prove the following theorem for the Friedrichs extension of $\pa$ from $\CI_c(\RRR)$. (See Section \ref{extension} for discussion of self-adjoint extensions.)
\begin{theorem*}
In polar coordinates, for $|\theta_1-\theta_2|\neq\pi$, the diffractive part of the Schwartz kernel of the sine Aharonov--Bohm propagator $\frac{\sin(t\sqrt{\pa})}{\sqrt{\pa}}$ is a polyhomogeneous conormal distribution with respect to $\{t=r_1+r_2\}\subset \RR_t \times (\RR^2 \backslash \{0\})\times (\RR^2 \backslash \{0\})$ with the integral expression:
\begin{equation*}
     E_{\emph{D}}(t,r_1,r_2,\theta_1,\theta_2)=\int_{\RR}e^{i\lambda(t-r_2-r_1)}a(t,r_1,r_2,\theta_1,\theta_2;\lambda)d\lambda.
\end{equation*}
Here, the amplitude $a$ has the asymptotic expansion:
\begin{equation*}
    a\sim-\frac{\sin\pi\alpha}{2\sqrt{r_1r_2}}\cdot\frac{e^{-i\theta_1}+e^{i\theta_2}}{\cos\theta_1+\cos\theta_2}\cdot\lambda^{-1}+\mathcal{O}(\lambda^{-2})\in S^{-1}\text{ as } \lambda\rightarrow\infty,
\end{equation*}
where its leading order term is also known as the \emph{diffraction coefficient}. $S^{-1}$ is the classical symbol class of order $-1$. 
\end{theorem*}

A formal computation based on the results of \cite{ford2017diffractive} and \cite{yang2020diffraction} also yields the same diffraction coefficient. We prove this result using a technique similar to the ``differentiating the cone point" developed in \cite{ford2018wave} by Ford, Hassell and Hillairet. The idea is that the geometric wave is invariant under the translation of the solenoid, therefore differentiating along the translation yields a purely diffractive wave. However, the presence of the vector potential prevents us from translating the solenoid directly; it turns out that introducing an extra term corresponding to the vector potential into the translation vector field enable us to generalize the method to our situation. In particular, it corresponds to the covariant derivative on the $U(1)$-bundle over $\RRR$.

The spectral and scattering theory of the Aharonov--Bohm Hamiltonian has been widely studied. The single solenoid case is completely solvable using a mode-by-mode decomposition method (see \cite{aharonov1959significance} \cite{ruijsenaars1983aharonov} \cite{adami1998aharonov} \cite{stovicek1998aharonov} \cite{pankrashkin2011spectral} for one infinitely thin solenoid and \cite{de2008mathematical} \cite{de2010scattering} for one solenoid with finite width), although from the perspective of wave propagation it is non-trivial to see how singularities of the solution to the wave equation propagate if one uses the infinite sum of mode solutions directly, and the computation of propagation of singularities of the Aharonov--Bohm Hamiltonian seems not to have been carried out. The scattering theory of the Aharonov--Bohm Hamiltonian with two or multiple solenoids was widely studied by Alexandrova--Tamura\cite{alexandrova2011resonance} \cite{alexandrova2014resonances}, Ito--Tamura\cite{ito2001aharonov}, Mine\cite{mine2005aharonov}, Nambu\cite{nambu2000aharonov},  {\v{S}}t'ov{\'\i}{\v{c}}ek\cite{vstovivcek1991scattering} \cite{vstovivcek1991krein} \cite{vstovivcek1994scattering}, Tamura\cite{tamura2007semiclassical} along with many other authors.

The novelty of this paper is the following: To the best of our knowledge, this is the first description of the Aharonov--Bohm wave propagator that is not using mode-by-mode solutions, which enables us, for the first time, to understand the propagation of singularities (cf. \cite[Section 6.1]{duistermaat1972fourier}) of the Aharonov--Bohm Hamiltonian. Unlike mode-decomposition results, our result can be applied to diffraction by multiple solenoids, using finite speed of propagation. In fact, in a recent project \cite{yang2021trace}, the author uses the result to compute singularities of the wave trace and obtains a lower bound for the number of resonances for the scattering problem with multiple solenoids.


\subsection*{Acknowledgement}
The author is grateful to Luc Hillairet and Jared Wunsch for proposing this topic, and to Dean Baskin, Luc Hillairet and two anonymous referees for helpful comments on the manuscript. The author is especially very grateful to Jared Wunsch for many helpful discussions as well as valuable comments on this manuscript.  

\section{Preliminaries}
\label{background}
In this section, we present some preliminary backgrounds on the Aharonov--Bohm wave propagator and the diffraction phenomenon. 

The operator we are interested in is the electromagnetic Hamiltonian 
\begin{equation}
    \pa=\left(\frac{1}{i}\nabla-\Vec{A}\right)^2
\end{equation}
with $\Vec{A}=-\alpha\left(-\frac{y}{x^2+y^2}, \frac{x}{x^2+y^2} \right)^T$ an electromagnetic vector potential and $\alpha$ depending on the field strength of the magnetic field in the solenoid at $r=0$. Note that this is a positive symmetric operator defined on $\CI_c(\RRR)\subset L^2(\RR^2)$. The magnetic vector potential corresponds to the case of an infinitely thin single solenoid placed at the origin, and it can generate the so-called Aharonov--Bohm effect\cite{aharonov1959significance}, which suggest that the electromagnetic vector potential is somehow more ``physical" that the magnetic field. Note the magnetic potential chosen here is gauge invariant under the addition of a curl free vector field, and it satisfies both Coulomb and Lorenz gauge conditions. 

We also define the wave operator corresponding to this electromagnetic Hamiltonian:
\begin{equation*}
    \Box=D_t^2-\pa,
\end{equation*}
where $D_t=\frac{1}{i}\partial_t$. 

Using the polar coordinates, the Hamiltonian can be written as 
\begin{equation*}
    \pa=D_r^2-\frac{i}{r}D_r+\frac{1}{r^2}\left(D_{\theta}+\alpha\right)^2,
\end{equation*}
where $D_r=\frac{1}{i}\partial_{r}$ and $D_{\theta}=\frac{1}{i}\partial_{\theta}$. This is now analogous to the Laplacian on cones which can be treated by using the separation of variables and the conic functional calculus by Cheeger--Taylor\cite{cheeger1982diffraction}; also see \cite{watson1995treatise} for Hankel transforms. Therefore, the Schwartz kernel of the wave propagator $W(t)=\frac{\sin(t\sqrt{\pa})}{\sqrt{\pa}}$ can be written as  
\begin{equation}
    W(t)=\sum_{k\in\mathbb{Z}}\int_0^{\infty}\frac{\sin(t\lambda)}{\lambda}J_{\nu_k}(\lambda r_1)J_{\nu_k}(\lambda r_2)\lambda e^{ik(\theta_1-\theta_2)}d\lambda
\end{equation}
where $\nu_k=|k+\alpha|$. Here we confuse the propagator with its Schwartz kernel. Without loss of generality, we assume $\alpha\in(0,1)$ since the shift of integers in $\alpha$ only corresponds to the shift between different eigenmodes $\varphi_k=e^{ik\theta}$ with the same coefficients. 

Now we give a brief introduction to the diffractive geometry. For a more detailed presentation we refer to \cite{melrose2004propagation} and \cite{ford2017diffractive}.

Consider the diffraction with respect to the solenoid at the origin. There are two types of broken geodesics passing through the solenoid, see Figure \ref{dgg}, which corresponds to two types of waves emanating from the solenoid after diffraction:

\begin{definition} 
Suppose $\gamma: (-\epsilon, +\epsilon)\rightarrow \RR^2$ is a continuous piecewise \emph{geodesic} on $\RR^2$ arriving at the solenoid only at time $t=0$, then: \begin{itemize}
\item The curve $\gamma$ is a \emph{diffractive geodesic} if the intermediate terminal point $\gamma(0_-)$ and the initial point $\gamma(0_+)$ satisfies
$$\gamma(0_-)=\gamma(0_+)=0.$$
\item The curve $\gamma$ is a \emph{geometric geodesic} if it is a straight line passing through $0$ at $t=0$.
\item The curve $\gamma$ is a \emph{strictly diffractive geodesic} if it is a diffractive geodesic but not a geometric geodesic.
\end{itemize}
\end{definition}
The geometric geodesics are those that are locally realizable as limits of geodesics in $\RRR$, i.e., straight lines passing through the origin.

\begin{figure}[H]
  \includegraphics[width=0.6\linewidth]{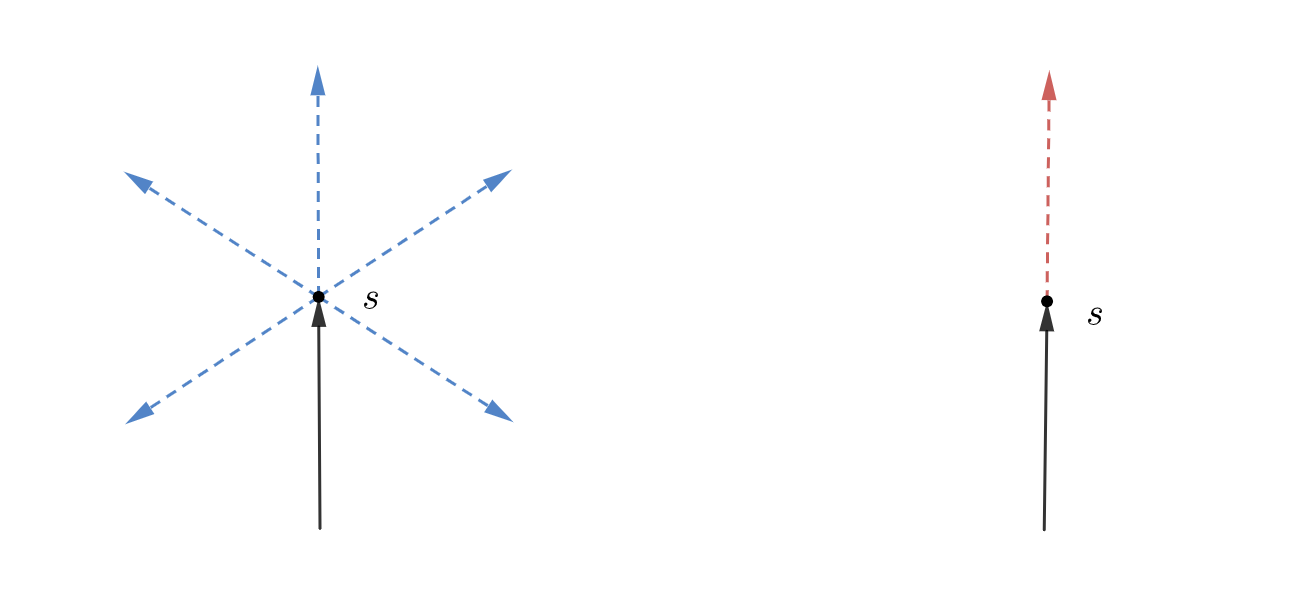}
  \caption{Diffractive and geometric geodesics}
  \medskip
  \small
    The blue rays on the left (together with the incoming ray) are diffractive geodesics at a solenoid $s$, while the red ray on the right is a geometric geodesic passing through $s$.
  \label{dgg}
\end{figure}

On smooth manifolds, singularities of the wave equation propagate along geodesics following standard propagation of singularities \cite{duistermaat1972fourier}. However, due to the presence of the solenoid, there are two different types of singularities--the diffractive and geometric wavefront--which propagate along the corresponding generalized geodesics defined above. Using polar coordinates $q_i=(r_i,\theta_i),i=1,2$, they are defined as two conic Lagrangian submanifolds $\Lambda^{\text{D}}$ and $\Lambda^{\text{G}}$ on $(0,\infty)_t\times T^*(\RRR)\times T^*(\RRR)$ as follows:
\begin{itemize}
    \item The diffractive wavefront is
    \begin{equation}
    \nonumber
    \Lambda^{\text{D}}\df N^*\{t=r_1+r_2\}
    \end{equation}
    where $N^*Y$ denotes the conormal bundle of the submanifold $Y$; 
    \item The geometric wavefront is
    \begin{equation}
    \nonumber
    \Lambda^{\text{G}}\df N^*\{t=|q_1-q_2|\};
    \end{equation}
    \item We also define their intersection 
    \begin{equation}
    \nonumber
    \Sigma\df \Lambda^{\text{D}}\cap \Lambda^{\text{G}}= \Lambda^{\text{D}}\cap \{|\theta_1-\theta_2|=\pi\}.
    \end{equation}
\end{itemize}
Therefore, modulo smooth remainders and away from $\Sigma$, the Schwartz kernel $E$ of the distribution $\frac{\sin(t\sqrt{\pa})}{\sqrt{\pa}}\delta_{q_2}$ can be decomposed into two parts:
\begin{equation}
    E(t)\equiv E_{\text{D}}(t)+ E_{\text{G}}(t).
\end{equation}
More precisely, the singularities of $E_{\text{D}}(t)$ are at $\Lambda^{\text{D}}$, while the singularities of $E_{\text{G}}(t)$ are at $\Lambda^{\text{G}}$. We define $E_{\text{D}}(t)$ to be the \emph{diffractive propagator} and $E_{\text{G}}(t)$ to be the \emph{geometric propagator}. For any fixed $t$ and $q_2$, the projection of singularities to the base space is described in Figure \ref{dgw}.

\begin{figure}[H]
  \includegraphics[width=0.5\linewidth]{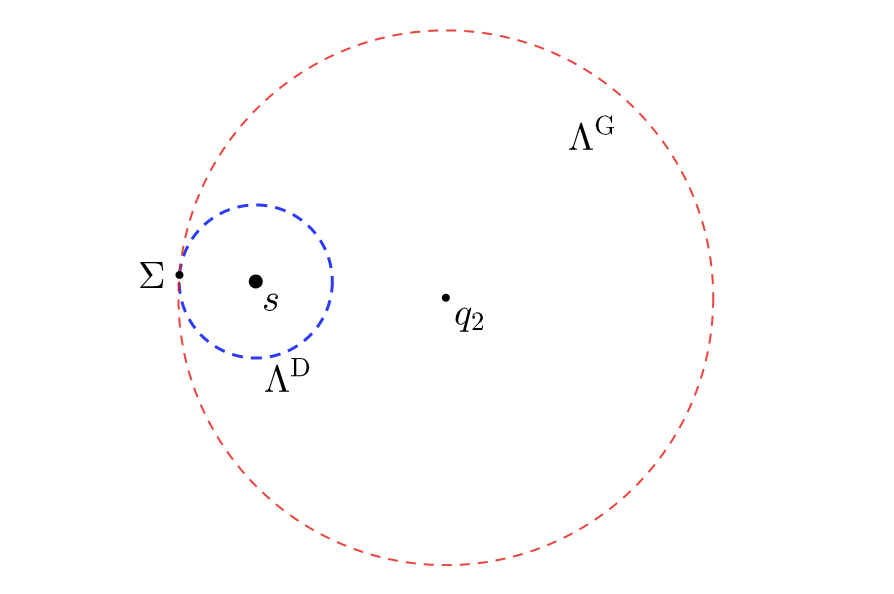}
  \caption{Diffractive and geometric wavefront}
  \medskip
  \small
    Consider a spherical wave emanating from $q_2$. The red circle is the geometric wavefront projected to the base; the blue circle is the diffractive wavefront projected to the base. The projection of their intersection $\Sigma$ is a single point.
  \label{dgw}
\end{figure}

In this paper, we focus our study on the diffractive propagator away from $\Sigma$, i.e., we consider $E_{\text{D}}(t,r_1,r_2,\theta_1,\theta_2)$ with $|\theta_1-\theta_2|\neq \pi$.

\section{The Friedrichs extension and domains}
\label{extension}
The self-adjoint extensions of the Aharonov--Bohm Hamiltonian on $\RR^2$ were studied in \cite{adami1998aharonov} and \cite{stovicek1998aharonov} using the method of deficiency subspaces. Since then they have been investigated by many authors. It was later also studied in \cite{derezinski2011homogeneous} \cite{derezinski2017schrodinger} \cite{derezinski2018radial} \cite{derezinski2020radial} as an application of Schr\"odinger operators with inverse square potential on the half-line and in \cite{correggi2018hamiltonians} \cite{correggi2021magnetic} as the Hamiltonian of a system consisting of two anyons. For simplicity and physical reasons, we focus our discussion on the Friedrichs extension of the magnetic Hamiltonian in the main text of this paper. In this section, we briefly discuss the Friedrichs extension of the Hamiltonian using the theory of deficiency indices for later purposes; we focus on the asymptotic behavior of the self-adjoint extensions near $0$. Similar results have previously been established in \cite{exner2002generalized} \cite{mine2005aharonov}; we include this material for the sake of completeness.

Consider the domain $\mathcal{D}(\pa)=\CI_c(\RR^2\backslash\{0\})$ of the positive symmetric operator $\pa$. We first define the closure $\bar{\pa}$ of $\pa$ under the graph norm:
$$\|u\|_{\pa}\df\|u\|_{L^2}+\|\pa u\|_{L^2},$$ 
and denote the domain of the operator closure $\bar{\pa}$ by $\mathcal{D}(\bar{\pa})$.
From \cite[Proposition 3.6]{gil2003adjoints}, we have a characterization\footnote{Note that we use the standard $L^2$-weight here rather that the $b$-weight used in \cite{gil2003adjoints}. This gives a one order difference in $r$ comparing to \cite{gil2003adjoints}.}:
\begin{equation}
\label{closure}
    \mathcal{D}(\bar{\pa})=r^2H^2_b(\RRR),
\end{equation}
which is independent of $\alpha$, where $H^2_b(\RRR)$ stands for b-Sobolev space defined by 
$$f\in H^2_b(\RRR)\Longleftrightarrow V_1V_2f\in L^2(\RRR)\text{ for all } V_1,V_2\in\mathcal{V}_b(\RRR),$$
where $\mathcal{V}_b=\text{span}_{\CI(\RR^2)}\{r\partial_r,\partial_{\theta}\}$.


Now we construct the Friedrichs extension of $\pa$. We write the Friedrichs extension of $\pa$ as $\pa^{\text{Fr}}$, and it is the unique self-adjoint extension whose domain is contained in the form domain:
$$\mathcal{D}(\pa^{1/2}):=\left\{u\in L^2(\RR^2): \|u\|_{L^2}+\|\pa^{1/2}u\|_{L^2}<\infty\right\},$$
where $\|\pa^{1/2}u\|_{L^2}$ is the norm induced by the quadratic form: 
$$q(u,v):=(\pa u,v)$$
in the sense of distributions.
The domain of the adjoint $\bar{\pa}^*$ is
$$\mathcal{D}(\bar{\pa}^*)=\left\{u\in L^2(\RR^2):\|u\|_{L^2}+\|\pa u\|_{L^2}<\infty\right\}$$
where $\pa$ acts distributionally. Following the theory of deficiency subspaces, we have the relation:
\begin{equation}
\label{deficiency}
    \mathcal{D}(\bar{\pa}^*)=\mathcal{D}(\bar{\pa})\oplus\mathcal{K}_-\oplus\mathcal{K_+},
\end{equation}
where $\mathcal{K}_-$ and $\mathcal{K_+}$ are deficiency subspaces with deficiency indices $(2,2)$ (c.f. \cite{adami1998aharonov}, \cite{stovicek1998aharonov}) defined by 
$$\mathcal{K}_{\pm}\df\text{Ker}(\bar{\pa}^*\mp i).$$ 
Therefore, every self-adjoint extension of $\bar{\pa}$ corresponds to a unitary map $U: \mathcal{K_+}\rightarrow\mathcal{K}_-.$ The following lemma determines the domain of the adjoint operator $\mathcal{D}(\bar{\pa}^*)$:
\begin{lemma}
\label{adjoint domain}
Let $\rho\in\CI_c((\RR_{+,r})$ be a smooth cutoff function with $\rho\equiv1$ for $r\leq1$. Then the domain $\mathcal{D}(\bar{\pa}^*)$ is 
\begin{equation}
    \mathcal{D}(\bar{\pa}^*)=\mathcal{D}(\bar{\pa})\oplus\text{span}_{\mathbb{C}}\{r^{\alpha}\rho,r^{-\alpha}\rho, r^{1-\alpha}e^{-i\theta}\rho, r^{\alpha-1}e^{-i\theta}\rho\}
\end{equation}
\end{lemma}
\begin{proof}
Following \eqref{deficiency}, we only need to characterize the deficiency subspaces. Consider the equations 
\begin{equation}
    (\pa+\beta_{\pm}^2)u=0
\end{equation}
where $\beta_{\pm}=e^{\mp i\pi/4}$. In polar coordinates, the equation becomes
\begin{equation}
    \left(-\partial_r^2-\frac{1}{r}\partial_r+\frac{1}{r^2}\left(D_{\theta}+\alpha\right)^2+\beta^2\right)u=0
\end{equation}
Standard separation and change of variables reduce this equation to a Bessel equation, and its solution is a linear combination of modified Bessel functions:
\begin{equation}
    u=\sum_{k\in\mathbb{Z}}\left(C_{1,k}I_{|k+\alpha|}(\beta r)+C_{2,k}K_{|k+\alpha|}(\beta r)\right)e^{ik\theta}.
\end{equation}
From \cite[10.25.3; 10.30.4]{NIST:DLMF}, $I_{|k+\alpha|}(\beta r)$ fails to be in $L^2(\RR_+,rdr)$ for all $k\in\mathbb{Z}$ due to the asymptotic behavior at $r\rightarrow\infty$, so there are only $K_{|k+\alpha|}(\beta r)$ terms that can be in the solutions. On the other hand, since $\alpha\in(0,1)$, the asymptotic behavior at $r\rightarrow0$ \cite[10.30.2]{NIST:DLMF} determines that $K_{|k+\alpha|}(\beta r)\in L^2(\RR_+,rdr)$ only for $k=-1,0$. Therefore, the deficiency subspaces are 
\begin{equation}
    \mathcal{K}_-\oplus\mathcal{K_+}=\text{span}_{\mathbb{C}}\{K_{\alpha}(\beta_{\pm} r),K_{1-\alpha}(\beta_{\pm} r)e^{-i\theta}\}.
\end{equation}
Let $\rho\in\CI_c(\RR_+)$ as in the statement of the lemma, and consider that both
\begin{equation*}
    [1-\rho(r)]K_{\alpha}(\beta_{\pm} r)\text{ and } [1-\rho(r)]K_{1-\alpha}(\beta_{\pm} r)
\end{equation*}
are Schwartz and vanish at $r=0$. Therefore, they are both in the space $\mathcal{D}(\bar{\pa})$. The domain of $\bar{\pa}^*$ thus can be written as
\begin{equation}
    \mathcal{D}(\bar{\pa}^*)=\mathcal{D}(\bar{\pa})\oplus \text{span}_{\mathbb{C}}\{\rho K_{\alpha}(\beta_{\pm} r),\rho K_{1-\alpha}(\beta_{\pm} r)e^{-i\theta}\}.
\end{equation}
The lemma then follows directly from the asymptotic expansions of the modified Bessel functions\cite[10.25.2; 10.27.4]{NIST:DLMF} when $r\rightarrow0$. This is because the only terms that are not in $\mathcal{D}(\bar{\pa})$ are the first two terms in their asymptotic expansions by the characterization \eqref{closure}, which exactly gives the terms with orders being $\pm\alpha$ and $\pm(1-\alpha)$ in $r$ for the corresponding mode. 
\end{proof}

Among all the self-adjoint extensions, in particular, the Friedrichs extension is the unique self-adjoint extension whose domain is contained in the form domain. For simplicity we let $P_{\alpha}$ denote the Friedrichs extension of the Aharonov--Bohm Hamiltonian henceforth. We define $\mathcal{D}_{s,\alpha}\df\mathcal{D}(\pa^{s/2})$. In particular, $\mathcal{D}_{2,\alpha}$ is the Friedrichs domain. We thus have the following lemma regarding the Friedrichs extension of the Aharonov--Bohm Hamiltonian:

\begin{lemma}
\label{Friedrichs}
For any $u\in\mathcal{D}_{2,\alpha}$, there exist constants $c_{-1}$ and $c_0$ in $\mathbb{C}$ and $v\in\mathcal{D}(\bP_{\alpha})$ such that
\begin{equation}
\label{F-expansion}
    u=\left(c_{-1}r^{1-\alpha}e^{-i\theta}+c_0r^{\alpha}\right)\rho(r)+v.
\end{equation}
\end{lemma}
\begin{proof}
The Friedrichs domain $\dom_{2,\alpha}$ is characterized as the subspace of $\mathcal{D}(\bar{\pa}^*)$ which is contained in the form domain $\dom(\pa^{1/2})$, i.e. any $u\in\dom_{2,\alpha}$ is  finite under the norm 
$$\|u\|_{L^2}+\|\pa^{1/2}u\|_{L^2}.$$
\cite[Lemma 8.1; Lemma 8.2]{gil2003adjoints} thus characterize this domain as 
$$\dom_{2,\alpha}=\mathcal{D}(\bar{\pa}^*)\cap r^{1-0}H^2_b(\RRR)$$
where $r^{1-0}H^2_b\df \bigcup_{\epsilon>0}r^{1-\epsilon}H^2_b$. Based on Lemma \ref{adjoint domain}, we thus conclude that any $u\in\dom_{2,\alpha}$ must take the form in \eqref{F-expansion}. 
\end{proof}

\begin{remark}
For later reference, when $\alpha\in(0,1)$, we also consider the domain $\dom_{2,-\alpha}$:
$$u\in\dom_{2,-\alpha}\Longleftrightarrow u=\left(c_{1}r^{1-\alpha}e^{i\theta}+c_0r^{\alpha}\right)\rho(r)+v,$$
where $v\in\mathcal{D}(\bP_{-\alpha})=r^2H^2_b(\RRR)$.
\end{remark}

\begin{remark}
For general non-Friedrichs self-adjoint extensions, we can obtain a two dimensional family of boundary asymptotics near the solenoid as $r\rightarrow0$:
\begin{equation*}
    u=\left(c_{1}^+r^{1-\alpha}e^{-i\theta}+c_0^+r^{\alpha}+c_{1}^-r^{\alpha-1}e^{-i\theta}+c_0^-r^{-\alpha}\right)\rho(r)+v
\end{equation*}
using the theory of the deficiency subspaces. In particular, these kinds of asymptotic behaviors can generate finitely many scattering poles/resonances due to a mode-mixing effect. This can be observed by examining the poles of the scattering matrix or the resolvent obtained in \cite{adami1998aharonov} \cite{stovicek1998aharonov}.
\end{remark}

We define $L_0, L_{-1}\in\mathcal{D}_{-2,\alpha}$ which map $u$ in the Lemma \ref{Friedrichs} to its corresponding coefficients $c_{-1}$ and $c_{0}$, where $\mathcal{D}_{-2,\alpha}$ is defined to be the dual space of $\dom_{2,\alpha}$ corresponding to the complex sesquilinear product $\la\cdot,\cdot\ra$, i.e., $\dom_{-2,-\alpha}$ is the dual space corresponding to the real bilinear product. The definitions of $L_0$ and $L_{-1}$ can be realized using the angular projection as follows: 
\begin{equation}
    \la L_0,u\ra\df\frac{1}{\sqrt{2\pi}}\lim_{r\downarrow0}\frac{1}{r^{\alpha}}\left[\Pi_0 u\right](r)\ 
    \text{  and  }\ 
    \la L_{-1},u\ra\df\frac{1}{\sqrt{2\pi}}\lim_{r\downarrow0}\frac{1}{r^{1-\alpha}}\left[\Pi_{-1}u\right](r)
\end{equation}
where 
$$\left[\Pi_ju\right](r)=\frac{1}{\sqrt{2\pi}}\int_0^{2\pi}u(r,\theta)e^{-ij\theta}d\theta$$
for $j=0,-1$.
We remark here that $L_0$ and $L_{-1}$ do not depend on choice of $\rho$ and they are supported at $0$. Here, we define $L\in\mathcal{D}_{-2,\alpha}$ to be \emph{supported at 0} if for any $u\in\dom_{2,\alpha}$, the pairing $\langle L, \varphi u\rangle$ vanishes for all $\varphi\in\CI_c(\RR^2)$ supported away from zero. Thus we have the following lemma regarding $L_j$ for $j=0,-1$: 
\begin{lemma}
\label{distribution_L}
Suppose $L\in\mathcal{D}_{-2,\alpha}$ and is supported only at $0$. Then $$L\in\text{span}\{L_{0},L_{-1}\}.$$ 
\end{lemma}
\begin{proof}
Take $u\in \mathcal{D}_{2,\alpha}$ with $\la L_j,u\ra=c_j$ for $j=0,-1$. Using the equation \eqref{F-expansion} which characterizes the Friedrichs extension, 
$$\la L,u\ra=\la L,v\ra+c_0\la L,r^{\alpha}\rho\ra+c_{-1}\la L,r^{1-\alpha}e^{-i\theta}\rho\ra$$
by Lemma \ref{Friedrichs}. In the above equation, $\la L,v\ra=0$ since $L$ is supported only at 0 and $v\in\dom(\bP_{\alpha})$ which is defined as the closure of $\CI_c(\RRR)$. Thus by $c_j=\la L_j,u\ra$,
\begin{equation}
    L=\la L,r^{\alpha}\rho\ra\cdot L_0+\la L,r^{1-\alpha}e^{-i\theta}\rho\ra\cdot L_{-1}
\end{equation}
as we claimed.
\end{proof}

\section{Differentiating the solenoid location}
If we fix two points $q_1,q_2\in\RR^2\backslash\{0\}$, and consider the geometric and diffractive fundamental solutions with respect to these two points; heuristically, moving the location of the solenoid will only change the diffractive wave without affecting the geometric wave except at the intersection of the two fronts. In other words, if we differentiate the wave propagator with respect to the location of the solenoid, we should have a purely diffractive wave. This idea was employed by Ford, Hassell, Hillairet in \cite{ford2018wave} to compute the structure of the wave propagator on the Euclidean surface with conic singularities. We shall make this technique mathematically rigorous for the Aharonov--Bohm propagator in this section.

We can always find a direction in which the translation of the solenoid does not pass through the line segment connecting $q_1$ and $q_2$. By the rotational symmetry, without loss of generality we can assume the translation is along $(-\infty,0]$ from the origin in the direction of the negative $x$-axis and $q_1=(r_1,\theta_1),q_2=(r_2,\theta_2)$ with $\theta_1,\theta_2\in(-\frac{\pi}{2},\frac{\pi}{2})$. This is equivalent to taking a branch cut at $(-\infty,0]$ of the complex plane and letting $q_1,q_2$ lie on the right side of the imaginary axis. Moving the solenoid along the negative $x$-axis corresponds to translating the points $q_1,q_2$ under the flow of $\varphi_X^s$ of a constant vector field $X=\partial_x$ in $\RR^2$ with the solenoid fixed. See Fig \ref{pic1} for a picture of translation by the vector field $X$.
\begin{remark}
\label{branchcut}
Here we only consider the angle $\theta_1,\theta_2\in (-\frac{\pi}{2},\frac{\pi}{2})\subset(-\pi,\pi)$ which can be achieved by the rotational symmetry. This also makes the angle function $\theta=\arctan(y/x)$ well-defined. Otherwise, the translated solenoid will be collinear with $(q_1,q_2)$ at some point $s=s_0$; the geometric wave will experience a phase shift there which corresponds to a changing of branch cut.
\end{remark}
Now we consider the translation. $(t,q_1,q_2)$ is translated by the flow of $X$ at time $s$ to 
$$\Phi^s_X(t,q_1,q_2)\df\left(t,\varphi_X^s(q_1),\varphi_X^s(q_2)\right)=(t,x_1+s,y_1,x_2+s,y_2).$$

Note that $P_{\alpha}$ is not translation invariant under $X$ due to the presence of the vector potential. Therefore, we instead consider the flow $\varphi_{T_x}^s$ under the twisted translation operator 
$$T_x=\partial_x-i\alpha\frac{y}{x^2+y^2}$$
of $X$, where $\frac{y}{x^2+y^2}=-\partial_x(\arctan(y/x))$ is induced by the vector potential $\Vec{A}$. 

\begin{figure}[H]
  \includegraphics[width=0.75\linewidth]{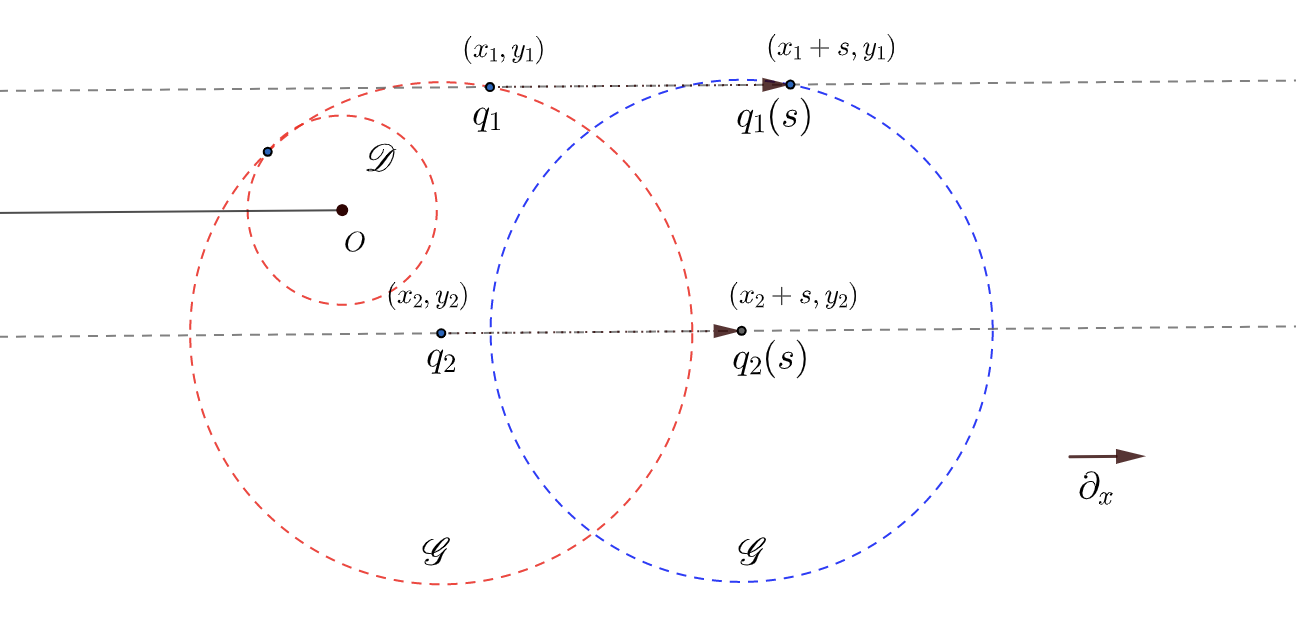}
  \caption{Moving solenoid technique}
  \medskip
  \small
  This is a translation of $q_1,q_2$ along the vector field $\partial_x$ by distance $s$. $O$ denotes the location of the solenoid; the red circles are the geometric ($\mathscr{G}$) and diffractive ($\mathscr{D}$) wavefront for $t=0$; the blue circle is the wavefront for $t=s$, which only has the geometric ($\mathscr{G}$) wavefront. 
  \label{pic1}
\end{figure}

\begin{remark}
The Aharonov--Bohm Hamiltonian has the following relation with the Laplacian on $\RRR$: 
\begin{equation}
\label{P_conjugation}
    \pa u(r,\theta)=e^{-i\alpha\theta}\Lap e^{i\alpha\theta}u(r,\theta),
\end{equation}
which could be treated using a $U(1)$-connection on the logarithmic covering of $\RRR$ and the usual Laplacian. In the following, we choose to use the operator $T_x$ directly to compute the diffractive wave of Aharonov--Bohm Hamiltonian $\pa$.
\end{remark}

We denote by $E(t,q_1,q_2)$ the fundamental solution to the wave equation, which is the Schwartz kernel of the wave propagator $W(t)$ defined in Section 2. First we consider the flow out induced by $T_x$ on functions 
$$ \left[(\Phi^s_{T_x})^*u\right](x,y)=u(x+s,y)e^{i\alpha\left(\arctan\left(\frac{y}{x+s}\right)-\arctan\left(\frac{y}{x}\right)\right)}. $$
Applying it to the fundamental solution gives
\begin{equation}
\label{trans_E}
    \left[(\Phi^s_{T_x})^*E\right](t,q_1,q_2)=E(t,x_1+s,y_1,x_2+s,y_2)e^{i\alpha\Delta\theta},
\end{equation}
where 
\begin{equation*}
\begin{split}
    \Delta\theta
    &=\Delta\theta_1-\Delta\theta_2\\
    &=\left(\arctan\left(\frac{y_1}{x_1+s}\right)-\arctan\left(\frac{y_1}{x_1}\right)\right)-\left(\arctan\left(\frac{y_2}{x_2+s}\right)-\arctan\left(\frac{y_2}{x_2}\right)\right)
\end{split}
\end{equation*}
is the total angle change of $q_1$ and $q_2$ by moving the solenoid under $\Phi_X^s$.

Now we consider the differentiation of the propagator with respect to the translation: 
\begin{equation}
\label{diff_trans_pgtr}
    \Upsilon_s\left(t,q_1,q_2\right)\df\partial_s\left[(\Phi^s_{T_x})^*E\right](t,q_1,q_2).
\end{equation}
When $s=0$, we calculate 
\begin{equation}
    \Upsilon_0\left(t,q_1,q_2\right)= X_1E(t,q_1,q_2)+X_2E(t,q_1,q_2)+E(t,q_1,q_2)\cdot\partial_s(i\alpha\Delta\theta)|_{s=0}
\end{equation}
where $X_j$ denotes $X$ acting on $q_j$ for $j=1,2$. Now we pair $\Upsilon_0\left(t,q_1,q_2\right)$ with a test function $\psi(q_2)$ to examine the differentiated propagator: 
\begin{equation*}
\begin{split}
    \la\Upsilon_0,\psi\ra_{q_2}&=\la X_1E,\psi\ra_{q_2}+\la X_2E,\psi\ra_{q_2}+\la E\cdot\partial_s(i\alpha\Delta\theta)|_{s=0},\psi\ra_{q_2}\\
                               &=\la X_1E,\psi\ra_{q_2}-\la E,X_2\psi\ra_{q_2}+\partial_s(i\alpha\Delta\theta_1)\la E,\psi\ra_{q_2}-\la E\cdot\partial_s(i\alpha\Delta\theta_2),\psi\ra_{q_2}\\
                               &=\left(X\circ W(t)\right)\psi-\left(W(t)\circ X\right)\psi-[i\alpha\frac{y}{x^2+y^2},W(t)]\psi\\
                               &=[T_x,W(t)]\psi.
\end{split}
\end{equation*}
The computation above then shows $\Upsilon_0(t,q_1,q_2)$ is the Schwartz kernel of the commutator $[T_x,W(t)]$. Using the operator identity
\begin{equation}
    \begin{split}
        \Box\circ \left[T_x,W(t)\right]&=\Box\circ T_x \circ W(t)-\Box\circ W(t)\circ T_x\\
        &=-\left[\pa,T_x\right]\circ W(t)+T_x\circ\Box\circ W(t)-\Box\circ W(t)\circ T_x\\
        &=\left[T_x,\pa\right]\circ W(t)
    \end{split}
\end{equation}
and Duhamel's principle, we have 
\begin{equation}
\label{Duhamel}
    \left[T_x,W(t)\right]= -\int_0^t W(t-s)\circ [T_x,\pa]\circ W(s)ds. 
\end{equation}
Thus, in order to understand $\Upsilon_0$, we need first to study the commutator $[T_x,\pa]$. 

Since $\dom_{2,\alpha}\subset H^1(\RR^2)$, we have $T_x: \dom_{2,\alpha}\rightarrow L^2(\RR^2)$. Thus for any $u\in\dom_{2,\alpha}$, applying the commutator gives that $[T_x,\pa]u\in\dom_{-2,\alpha}$. On the other hand, if $u$ is compactly supported in $\RR^2\backslash\{0\}$, then we have $[T_x,\pa]u=0$. Therefore, the distribution $[T_x,\pa]u$ for any $u\in\dom_{2,\alpha}$ can only be supported at $0$, and we can apply Lemma \ref{distribution_L} to show $[T_x,\pa]u$ is a linear combination of $L_0$ and $L_{-1}$. Also note that since 
\begin{equation}
\label{conjugation_Tx}
T_x=e^{-i\alpha\theta}\partial_x e^{i\alpha\theta},    
\end{equation}
combining with equation \eqref{P_conjugation}, the commutator satisfies
\begin{equation}
\label{commutator_TxP}
    [T_x,\pa]=e^{-i\alpha\theta}[\partial_x,\Lap] e^{i\alpha\theta}.
\end{equation}
Similarly for $T_y=\partial_y+i\alpha\frac{x}{x^2+y^2}$, where $i\alpha\frac{x}{x^2+y^2}=\partial_y(\arctan(y/x))$, we have the commutator 
\begin{equation}
\label{commutator_TyP}
    [T_y,\pa]=e^{-i\alpha\theta}[\partial_y,\Lap] e^{i\alpha\theta},
\end{equation}
and the discussion above relating to $T_x$ also applies to $T_y$ without any changes. Now we use these conjugation relations to compute the distribution $[Q,\pa]u$ for $u\in\dom_{2,\alpha}$ and $Q\in\text{span}_{\RR}\{T_x,T_y\}$. We use complex coordinates in the following proposition since it makes the proof more concise. 

\begin{proposition}
\label{distribution_QPu}
Let $Q\in\text{span}_{\RR}\{T_x,T_y\}$ be a twisted translation operator on $\RRR$. Then in complex coordinates 
\begin{equation}
\label{conjugation_Q}
  Q=e^{-i\alpha\theta}\left(X_z\partial_z+X_{\bz}\partial_{\bz}\right)e^{i\alpha\theta}  
\end{equation} 
for some $X_z,X_{\bz}\in\mathbb{C}$. For any $u\in\dom_{2,\alpha}$, the commutator 
\begin{equation}
\label{commutator_QP}
    [Q,\pa]u=-4\pi\alpha(1-\alpha)\cdot\left(X_{z}L_0(u)\cdot L_{-1}+X_{\bz}L_{-1}(u)\cdot L_0\right).
\end{equation}
\end{proposition}
\begin{proof}
Since $\partial_z=\frac{1}{2}(\partial_x-i\partial_y)$ and $\partial_{\bz}=\frac{1}{2}(\partial_x+i\partial_y)$, together with the equation \eqref{conjugation_Tx} we conclude the equation \eqref{conjugation_Q}. 

Consider the sesquilinear pairing 
$$\la [Q,\pa]u,\tu\ra, \text{ for } u,\tu\in\dom_{2,\alpha}.$$
By the discussion above, the pair vanishes if either $u$ or $\tu$ is in $\dom(\bP_{\alpha})$. This is because that $[Q,\pa]$ is now self-adjoint and that $[Q,\pa]u$ is supported at $0$ for any $u\in\dom_{2,\alpha}$; pairing $[Q,\pa]u$ with any $\tu\in\dom(\bP_{\alpha})$ thus vanishes. It remains to consider only $u$ and $\tu$ to be the linear combination of 
$$u_0\df r^{\alpha}\rho(r) \text{ and } u_{-1}\df r^{1-\alpha}e^{-i\theta}\rho(r)$$ 
by Lemma \ref{Friedrichs}. A straightforward computation shows that 
$$\partial_z(e^{ik\theta})=\frac{k}{2r}e^{i(k-1)\theta}
\text{ and } 
\partial_{\bz}(e^{ik\theta})=-\frac{k}{2r}e^{i(k+1)\theta}$$
for any $k\in\RR$. Thus $$[Q,\pa]=e^{-i\alpha\theta}[X_z\partial_z+X_{\bz}\partial_{\bz},\Lap]e^{i\alpha\theta}$$
switches the modes correspondingly depending on the coefficients $X_z$ and $X_{\bz}$. Thus the only non-vanishing parts of the pairings are linear combinations of $\la e^{-i\alpha\theta}[\partial_{z},\Lap]e^{i\alpha\theta}u_0,u_{-1}\ra$ and $\la e^{-i\alpha\theta}[\partial_{\bz},\Lap]e^{i\alpha\theta}u_{-1},u_{0}\ra$. Thus we must have 
\begin{equation}
\label{commutator_QP2}
    [Q,\pa]u=C_1\cdot X_{z}L_0(u)\cdot L_{-1}+C_2\cdot X_{\bz}L_{-1}(u)\cdot L_0.
\end{equation} Now we compute the pairings to get $C_1$ and $C_2$. We only compute the first, the other pairing is the same. We define\footnote{Although these are only locally well-defined functions on $\RR^2$ individually, the pairing is well-defined due to the complex conjugation in the sesquilinear pairing.} $v_0:=e^{i\alpha\theta}u_0$ and $v_{-1}:=e^{i\alpha\theta}u_{-1}$, 
\begin{equation}
\label{pairing1}
    \begin{split}
        \la e^{-i\alpha\theta}[\partial_{z},\Lap]e^{i\alpha\theta}u_0,u_{-1}\ra
        &= \la [\partial_{z},\Lap]e^{i\alpha\theta}u_0,e^{i\alpha\theta}u_{-1}\ra\\
        &= \la [\partial_{z},\Lap]v_0,v_{-1}\ra\\
        &= -\la \Lap v_0,\partial_{\bz}v_{-1}\ra-\la \partial_{z} v_0,\Lap v_{-1}\ra\\
        &=-\int_{\RRR}\left(\Lap \bar{v}_0\cdot\partial_{\bz}v_{-1}+\partial_{\bz} \bar{v}_0\cdot\Lap v_{-1}\right)dxdy
    \end{split}
\end{equation}
Note that since $dxdy=-\frac{1}{2i}dzd\bz$ and  $\Lap=4\partial_z\partial_{\bz}$, we can apply Stokes' theorem to compute as following: 
\begin{equation}
    \begin{split}
        \eqref{pairing1}
        &=\frac{1}{2i}\lim_{\epsilon\rightarrow0}\int_{z\geq\epsilon}\left(\Lap \bar{v}_0\cdot\partial_{\bz}v_{-1}+\partial_{\bz} \bar{v}_0\cdot\Lap v_{-1}\right)dzd\bz\\
        &= \frac{2}{i}\lim_{\epsilon\rightarrow0}\int_{|z|=\epsilon}\left(\partial_{\bz} \bar{v}_0\cdot\partial_{\bz}v_{-1}\right)d\bz\\
        &=\frac{2}{i}\alpha(1-\alpha)\lim_{\epsilon\rightarrow0}\int_{|z|=\epsilon}\frac{1}{\bz}d\bz\\
        &=-4\pi \alpha(1-\alpha)
    \end{split}
\end{equation}
since $\bar{v}_0=\bz^{\alpha}\rho(|z|)$, $v_{-1}=\bz^{1-\alpha}\rho(|z|)$ and $\rho$ is equal to 1 for $|z|$ small enough. This leads to our conclusion \eqref{commutator_QP}.
\end{proof}

\begin{remark}
The idea of translating the solenoid only seems to apply to the Friedrichs extension. This is due to the fact that the Friedrichs extension corresponds to the most regular boundary condition as $r\rightarrow0$, i.e. $H^1(\RR^2)$, among all self-adjoint extensions. For general self-adjoint extensions, $[T_x,\pa]$ no longer maps the domain to its dual. In particular, failure to be in $L^2$ for other self-adjoint extensions under translation prevents us from carrying out the above computation of $\la[T_x,\pa]u,u'\ra$ for $u,u'\in\dom(\pa^{\emph{SA}})$, the domain of general self-adjoint extensions. 
\end{remark}

\section{The differentiated propagator}
We derive a formula for the differentiated Aharonov--Bohm propagator $\Upsilon_0(t,q_1,q_2)$ in this section. Note that since $T_x=e^{-i\alpha\theta}(\partial_z+\partial_{\bz})e^{i\alpha\theta}$, by combining equation \eqref{Duhamel} with Proposition \ref{distribution_QPu}, we conclude
\begin{equation}
\label{D_propagator}
\begin{split}
    \Upsilon_0(t,q_1,q_2)
    &= 4\pi \alpha(1-\alpha)\int_0^t W(t-s)\circ [L_0\circ W(s)\cdot L_{-1}+L_{-1}\circ W(s)\cdot L_0]ds\\
    &= 4\pi \alpha(1-\alpha)\int_0^t \Big\{[W(t-s) L_{-1}](q_1)\cdot [L_0\circ W(s)](q_2)\\
    &\ \ \ \ \ \ \ \ \ \ \ \ \ \ \ \ \ \ \ \ \ \ +[W(t-s) L_{0}](q_1)\cdot [L_{-1}\circ W(s)](q_2)\Big\} ds.
\end{split}
\end{equation}
We define 
$$l_j(t):=L_j\circ W(t)$$
for $j=-1,0$ by propagating (a test function) under the flow of $W(t)$ then applying $L_j, j=-1,0$. On the other hand, $W(t)L_j$ can be obtained from $l_j(t)$ through the adjoint of the wave propagator. Since $L_j$ is supported only at the solenoid, i.e., the origin, we should expect both $l_j(t)$ and $W(t)L_j$ to be like spherical waves emanating from the solenoid, i.e., purely diffractive waves. Indeed we have the following lemma: 
\begin{lemma}
\label{Propagated_distributions}
For $t>0$, the distributions $l_j(t),j=-1,0$ are given by: 
\begin{equation}
\label{l-1}
    l_{-1}(t)=\frac{1}{i2^{1-\alpha}\cdot\Gamma(2-\alpha)\sqrt{8\pi r}} \int_{\RR}e^{i\lambda(t-r)}e^{i(\frac{\pi}{2}(1-\alpha)+\frac{\pi}{4})}\lambda^{\frac{1}{2}-\alpha}\left(P(\lambda r)+iQ(\lambda r)\right)e^{i\theta_2}d\lambda,
\end{equation}
\begin{equation}
\label{l0}
    l_{0}(t)=\frac{1}{i2^{\alpha}\cdot\Gamma(\alpha+1)\sqrt{8\pi r}} \int_{\RR}e^{i\lambda(t-r)}e^{i(\frac{\pi}{2}\alpha+\frac{\pi}{4})}\lambda^{\alpha-\frac{1}{2}}(P(\lambda r)+iQ(\lambda r))d\lambda,
\end{equation}
where $P(\lambda r)$ and $Q(\lambda r)$ defined in the proof are polyhomogenous symbols of order $0$ with principal symbols equal to $1$. Thus they are polyhomogeneous conormal distributions with respect to $N^*\{t=r\}$.
\end{lemma}
\begin{proof}
We only compute $l_{-1}(t)$ here; the computation of $l_0(t)$ is similar. First recall that $L_{-1}$ can be written in terms of the angular spectral projector as 
$$\la L_{-1},u\ra=\frac{1}{2\pi}\lim_{r\downarrow0}\frac{1}{r^{1-\alpha}}\int_0^{2\pi}u(r,\theta)e^{i\theta}d\theta.$$
By the functional calculus on cones \cite{cheeger1982diffraction}, the kernel of the wave propagator $W(t)$ takes the form
\begin{equation*}
    E(t,r_1,\theta_1,r_2,\theta_2)=\sum_{j\in\mathbb{Z}}e^{ij(\theta_1-\theta_2)}\int_0^{\infty}\frac{\sin(\lambda t)}{\lambda}J_{|j+\alpha|}(\lambda r_1)J_{|j+\alpha|}(\lambda r_2)\lambda d\lambda.
\end{equation*}
Since $L_{-1}$ only involves the $-1$ mode, applying $L_{-1}$ to the propagated distribution $W(t)u$, we get
\begin{equation*}
\begin{split}
    [l_{-1}(t)]u=\lim_{r_1\downarrow0}\frac{1}{2\pi r_1^{1-\alpha}}\int_{\theta_1=0}^{2\pi}&\int_{r_2=0}^{\infty}\int_{\theta_2=0}^{2\pi}\left(\int_0^{\infty}\frac{\sin(\lambda t)}{\lambda}J_{1-\alpha}(\lambda r_1)J_{1-\alpha}(\lambda r_2)\lambda d\lambda\right)\\
    & u(r_2,\theta_2)e^{i\theta_2}r_2dr_2d\theta_2d\theta_1.
\end{split}
\end{equation*}
Integrating in $\theta_1$ we have 
\begin{equation*}
\begin{split}
    [l_{-1}(t)]u=\lim_{r_1\downarrow0}\frac{1}{ r_1^{1-\alpha}}\int_{r_2=0}^{\infty}&\int_{\theta_2=0}^{2\pi}\left(\int_0^{\infty}\frac{\sin(\lambda t)}{\lambda}J_{1-\alpha}(\lambda r_1)J_{1-\alpha}(\lambda r_2)\lambda d\lambda\right)\\
    &u(r_2,\theta_2)e^{i\theta_2}r_2dr_2d\theta_2.
\end{split}
\end{equation*}
Thus the Schwartz kernel of $l_{-1}$ is
\begin{equation}
\begin{split}
    l_{-1}(t)=\lim_{r_1\downarrow0}\frac{1}{ r_1^{1-\alpha}}\left(\int_0^{\infty}\frac{\sin(\lambda t)}{\lambda}J_{1-\alpha}(\lambda r_1)J_{1-\alpha}(\lambda r_2)\lambda d\lambda\right)e^{i\theta_2}.
\end{split}
\end{equation}
Using the asymptotic behavior of $J_{1-\alpha}(\lambda r_1)$ for $r\rightarrow0$ \cite[10.7.3]{NIST:DLMF}:
\[J_{\nu}\left(z\right)\sim(\tfrac{1}{2}z)^{\nu}/\Gamma\left(\nu+1\right),\]
we obtain
\begin{equation*}
\begin{split}
    l_{-1}(t)
    &=\lim_{r_1\downarrow0}\frac{1}{ r_1^{1-\alpha}}\left(\int_0^{\infty}\frac{\sin(\lambda t)}{\lambda}\frac{(\frac{1}{2}\lambda r_1)^{1-\alpha}}{\Gamma(2-\alpha)}J_{1-\alpha}(\lambda r_2)\lambda d\lambda\right)e^{i\theta_2}\\
    &=\frac{1}{2^{1-\alpha}\cdot\Gamma(2-\alpha)}\left(\int_0^{\infty}\sin(\lambda t)\lambda^{1-\alpha}J_{1-\alpha}(\lambda r_2) d\lambda\right)e^{i\theta_2}.
\end{split}
\end{equation*}
Consider the asymptotic behavior of $J_{1-\alpha}(\lambda r_2)$ as $r\rightarrow\infty$ using \cite[10.17.3]{NIST:DLMF}:
\[J_{\nu}\left(z\right)\sim\left(\frac{2}{\pi z}\right)^{\frac{1}{2}}\*\left(\cos\omega\sum_{k=0}^{\infty}(-1)^{k}\frac{a_{2k}(\nu)}{z^{2k}}-\sin\omega\sum_{k=0}^{\infty}(-1)^{k}\frac{a_{2k+1}(\nu)}{z^{2k+1}}\right),\]
where $\omega=z-\tfrac{1}{2}\nu\pi-\tfrac{1}{4}\pi$, $a_0(\nu)=1$ and  
\[a_{k}(\nu)=\frac{(4\nu^{2}-1^{2})(4\nu^{2}-3^{2})\cdots(4\nu^{2}-(2k-1)^{2})}{k!8^{k}}.\]
Then for the distribution
\begin{equation*}
\begin{split}
    \tilde{l}_{-1}(t):=\int_0^{\infty}\sin(\lambda t)\lambda^{1-\alpha}J_{1-\alpha}(\lambda r) d\lambda,
\end{split}
\end{equation*}
let us consider its leading singularities for simplicity, which can be extracted from the leading part of the asymptotic expansion of the Bessel function $J_{1-\alpha}(\lambda r_2)$. Thus modulo lower order singularities, we have 
\begin{equation*}
\begin{split}
    \tilde{l}_{-1}(t)
    &\equiv\int_0^{\infty}\frac{e^{i\lambda t}-e^{-i\lambda t}}{2i}\lambda^{1-\alpha}\left(\frac{2}{\pi\lambda r}\right)^{\frac{1}{2}}\left(\frac{e^{i(\lambda r-\frac{\pi}{2}(1-\alpha)-\frac{\pi}{4})}+e^{-i(\lambda r-\frac{\pi}{2}(1-\alpha)-\frac{\pi}{4})}}{2}\right) d\lambda\\
    &= \frac{1}{i\sqrt{8\pi r}} \int_0^{\infty}(e^{i\lambda t}-e^{-i\lambda t})\lambda^{\frac{1}{2}-\alpha}\left(e^{i(\lambda r-\frac{\pi}{2}(1-\alpha)-\frac{\pi}{4})}+e^{-i(\lambda r-\frac{\pi}{2}(1-\alpha)-\frac{\pi}{4})}\right) d\lambda\\
    &= \frac{1}{i\sqrt{8\pi r}} \left(\int_{\RR}e^{i\lambda(t+r)}\lambda^{\frac{1}{2}-\alpha}e^{-i(\frac{\pi}{2}(1-\alpha)+\frac{\pi}{4})}d\lambda+\int_{\RR}e^{i\lambda(t-r)}\lambda^{\frac{1}{2}-\alpha}e^{i(\frac{\pi}{2}(1-\alpha)+\frac{\pi}{4})}d\lambda\right)
\end{split}
\end{equation*}
where the last equation is obtained by changing signs for two of the four total integrands. The same procedure can be applied to the total singularities of $\tilde{l}_{-1}(t)$. We define 
$$P(\lambda r)\df\sum_{k=0}^{\infty}(-1)^{k}\frac{a_{2k}(\nu)}{(\lambda r)^{2k}} \text{ and } Q(\lambda r)\df\sum_{k=0}^{\infty}(-1)^{k}\frac{a_{2k+1}(\nu)}{(\lambda r)^{2k+1}},$$
where $\nu=1-\alpha$. Note that $P$ is even in $\lambda$ and $Q$ is odd in $\lambda$, by the same changing variables trick:
\begin{equation*}
\begin{split}
    \tilde{l}_{-1}(t)
    &= \frac{1}{i\sqrt{8\pi r}} \int_{\RR}\left(e^{i\lambda(t+r)}e^{-i(\frac{\pi}{2}(1-\alpha)+\frac{\pi}{4})}+e^{i\lambda(t-r)}e^{i(\frac{\pi}{2}(1-\alpha)+\frac{\pi}{4})}\right)\lambda^{\frac{1}{2}-\alpha}P(\lambda r)d\lambda\\
    &\ +\frac{1}{\sqrt{8\pi r}} \int_{\RR}\left(e^{i\lambda(t+r)}e^{-i(\frac{\pi}{2}(1-\alpha)+\frac{\pi}{4})}+e^{i\lambda(t-r)}e^{i(\frac{\pi}{2}(1-\alpha)+\frac{\pi}{4})}\right)\lambda^{\frac{1}{2}-\alpha}Q(\lambda r)d\lambda.
\end{split}
\end{equation*}
Fot $t>0$, this is a polyhomogeneous conormal distribution at $N^*\{t=r\}$ with symbol in $S^{\frac{1}{2}-\alpha}_{phg}$. Thus, we conclude the distribution $l_{-1}(t)$ takes the form: 
\begin{equation*}
\begin{split}
    l_{-1}(t)
    =\frac{1}{i2^{1-\alpha}\cdot\Gamma(2-\alpha)\sqrt{8\pi r}} \int_{\RR}e^{i\lambda(t-r)}e^{i(\frac{\pi}{2}(1-\alpha)+\frac{\pi}{4})}\lambda^{\frac{1}{2}-\alpha}\left(P(\lambda r)+iQ(\lambda r)\right)e^{i\theta_2}d\lambda
\end{split}
\end{equation*}
Similarly, the $l_0(t)$ takes the form:
\begin{equation*}
\begin{split}
    l_{0}(t)
    =\frac{1}{i2^{\alpha}\cdot\Gamma(\alpha+1)\sqrt{8\pi r}} \int_{\RR}e^{i\lambda(t-r)}e^{i(\frac{\pi}{2}\alpha+\frac{\pi}{4})}\lambda^{\alpha-\frac{1}{2}}(P(\lambda r)+iQ(\lambda r))d\lambda.
\end{split}
\end{equation*}
\end{proof}

\begin{lemma}
\label{W_tL}
The distributions $W(t)L_{-1}$ and $W(t)L_{0}$ take the forms: 
\begin{equation}
\nonumber
    W(t)L_{-1}= \frac{1}{i2^{1-\alpha}\cdot\Gamma(2-\alpha)\sqrt{8\pi r}} \int_{\RR}e^{i\lambda(t-r)}e^{i(\frac{\pi}{2}(1-\alpha)+\frac{\pi}{4})}\lambda^{\frac{1}{2}-\alpha}\left(P(\lambda r)+iQ(\lambda r)\right)e^{-i\theta}d\lambda,
\end{equation}
\begin{equation}
\nonumber
    W(t)L_{0}=\frac{1}{i2^{\alpha}\cdot\Gamma(\alpha+1)\sqrt{8\pi r}} \int_{\RR}e^{i\lambda(t-r)}e^{i(\frac{\pi}{2}\alpha+\frac{\pi}{4})}\lambda^{\alpha-\frac{1}{2}}(P(\lambda r)+iQ(\lambda r))d\lambda.
\end{equation}
\end{lemma}

\begin{proof}
Again we only consider $W(t)L_{-1}$; the other follows from the same argument. 

Consider the pairing $\la W(t)L_{-1}, \varphi\ra$. Note that $W(t)$ is a Hermitian operator, hence
$$\la W(t)L_{-1}, \varphi\ra=\la L_{-1},W(t)\varphi\ra=(L_{-1}\circ W(t))\varphi=l_{-1}(t)\varphi.$$
$W(t)L_{-1}$ must agree with the Schwartz kernel of $l_{-1}(t)$, except with the term $e^{i\theta}$ switched to $e^{-i\theta}$ due to the bracket being sesquilinear product. 
\end{proof}

Now based on the equation \eqref{D_propagator} and Lemma \ref{Propagated_distributions}, Lemma \ref{W_tL}, we conclude that $\Upsilon_0(t,q_1,q_2)$ is a polyhomogeneous Lagrangian distribution associated to the diffractive Lagrangian relations $N^*\{t=r_1+r_2\}$.
\begin{proposition}
\label{D_full_propagator}
For $t>0$, the differentiated propagator $\Upsilon_0(t,q_1,q_2)$ is given by
\begin{equation}
    \Upsilon_0(t,q_1,q_2)= \frac{\sin\pi\alpha}{4\pi\sqrt{r_1r_2}}\left(\int_{\RR}e^{i\lambda(t-r_2-r_1)}\tilde{a}(r_1,r_2,\lambda)d\lambda \right)(e^{-i\theta_1}+e^{i\theta_2}),
\end{equation}
for some $\tilde{a}\in S^0_{phg}$ with principal symbol equal to $1$.
\end{proposition}
\begin{proof}
Following equation \eqref{D_propagator}, we have: 
\begin{equation*}
    \Upsilon_0(t,q_1,q_2)= 4\pi \alpha(1-\alpha)\int_0^t \Big\{[W(t-s)L_{-1}](q_1)\cdot [l_0(s)](q_2)+[W(t-s)L_0](q_1)\cdot [l_{-1}(s)](q_2)\Big\}ds.
\end{equation*}
To begin, we consider the first integrand: 
\begin{equation*}
    \Upsilon_0^{(1)}(t,q_1,q_2):=4\pi \alpha(1-\alpha)\int_0^t [W(t-s)L_{-1}](q_1)\cdot [l_0(s)](q_2)ds.
\end{equation*}
Similarly, we define the second integrand as $\Upsilon_0^{(2)}$. Substituting $W(t-s)L_{-1}$ and $l_0(s)$ with equations from Lemma \ref{Propagated_distributions} and Lemma \ref{W_tL}, and defining $a(\lambda r):=P(\lambda r)+iQ(\lambda r)$, it becomes
\begin{equation*}
    \frac{\sin\pi\alpha}{4\pi\sqrt{r_1r_2}}\int_0^t \left(\int_{\RR}\int_{\RR}e^{i(\lambda(t-s-r_1)+\xi(s-r_2))}\lambda^{\frac{1}{2}-\alpha}\xi^{\alpha-\frac{1}{2}}a(\lambda r_1)a(\xi r_2)d\lambda d\xi\right)e^{-i\theta_1}ds.
\end{equation*}
Applying the stationary phase lemma in $(s,\xi)$, we conclude that it is a polyhomogeneous conormal distribution at $N^*\{t=r_1+r_2\}$: 
\begin{equation}
    \Upsilon_0^{(1)}(t,q_1,q_2)=\frac{\sin\pi\alpha}{4\pi\sqrt{r_1r_2}} \left(\int_{\RR}e^{i\lambda(t-r_2-r_1)}\tilde{a}(r_1,r_2,\lambda)d\lambda \right)e^{-i\theta_1},
\end{equation}
where $\tilde{a}(r_1,r_2,\lambda)$ has certain asymptotic expansion, which is one-step in $\lambda$. Similarly, for the second integrand, the same process gives
\begin{equation}
\Upsilon_0^{(2)}(t,q_1,q_2)=\frac{\sin\pi\alpha}{4\pi\sqrt{r_1r_2}}\left(\int_{\RR}e^{i\lambda(t-r_2-r_1)}\tilde{a}(r_1,r_2,\lambda)d\lambda \right)e^{i\theta_2}.
\end{equation}
Thus, the differentiated propagator
\begin{equation*}
    \Upsilon_0(t,q_1,q_2)= \frac{\sin\pi\alpha}{4\pi\sqrt{r_1r_2}}\left(\int_{\RR}e^{i\lambda(t-r_2-r_1)}\tilde{a}(r_1,r_2,\lambda)d\lambda \right)(e^{-i\theta_1}+e^{i\theta_2}),
\end{equation*}
which is a polyhomogeneous conormal distribution with respect to $N^*\{t=r_1+r_2\}$.
\end{proof}

\section{The Aharonov--Bohm wave propagator}
In this section, we compute the diffractive wave propagator $E_{\text{D}}(t)$ of the Aharonov--Bohm Hamiltonian. For the fundamental solution $E(t,r_1,\theta_1,r_2,\theta_2)$ of the Aharonov--Bohm Hamiltonian, the standard propagation of singularities argument gives the singularities at the geometric wavefront which is a spherical wave of radius $t$ from the source; the geometric theory of diffraction \cite{keller1962geometrical} together with the finite speed of propagation suggest that there are other singularities emanating from the solenoid within the ball of radius $r_1=t-r_2$ for $t>r_2$, which are the diffractive singularities. In this section, we show the diffractive wavefront are conormal singularities with respect to
$\{t=r_1+r_2\}$ and compute its amplitude.

For $|\theta_1-\theta_2|\neq\pi$, fix a finite time $t$. By the moving solenoid technique, for $s\gg1$ large enough, the solenoid is far away from the pair of points $(q_1,q_2)$. Thus, the translated diffractive wave propagator $\big[(\Phi^s_{T_x})^*E_{\text{D}}\big](t)\equiv 0$ for $s\gg1$ by finite speed of propagation, i.e., the equation $$\big[(\Phi^s_{T_x})^*E_{\text{G}}\big](t,q_1,q_2)\equiv \big[(\Phi^s_{T_x})^*E\big](t,q_1,q_2)$$
holds modulo smooth errors. (See Figure \ref{pic1}.) By the translation invariance of solutions under $T_x$ away from the solenoid, for $S,S'\gg 1$,
\begin{equation}
\label{trans_equiv}
    \big[(\Phi^S_{T_x})^*E\big](t,q_1,q_2)= \big[(\Phi^{S'}_{T_x})^*E\big](t,q_1,q_2).
\end{equation}
Therefore, by fundamental theorem of calculus and equation \eqref{diff_trans_pgtr}, for $|\theta-\theta'|\neq \pi$ the diffractive propagator
\begin{equation}
    E_{\text{D}}(t,q_1,q_2)\equiv E(t,q_1,q_2)-\big[(\Phi^S_{T_x})^*E\big](t,q_1,q_2)=-\int^{S}_0\Upsilon_s(t,q_1,q_2) ds
\end{equation}
modulo smooth terms, since the second part in the middle only has geometric singularities, which cancels out the geometric singularities in the first term. By \eqref{trans_equiv}, the above equation is independent of $S$; letting $S\rightarrow\infty$ yields
\begin{equation}
\label{ftc}
    E_{\text{D}}(t)\equiv -\int^{\infty}_0\Upsilon_s(t,q_1,q_2) ds
\end{equation}
modulo smooth remainders.

\begin{remark}
If we include $|\theta_1-\theta_2|=\pi$ into consideration, then in the equation \eqref{ftc} the geometric singularities in $E(t)$ and $\big[(\Phi^S_{T_x})^*E\big](t)$ cancel each other out. The cancellation of the geometric waves is due to the assumptions we made at the beginning of Section 4 together with the remark \ref{branchcut}, since the phase of the geometric wave remains unchanged under the twisted translation with these assumptions. On the other hand, if the solenoid and $q_1,q_2$ happen to be collinear in the twisted translation, i.e. the aforementioned assumptions fail, the geometric wave will experience a phase shift and the geometric waves in $E(t)$ and $\big[(\Phi^S_{T_x})^*E\big](t)$ will not cancel each other out. Under this circumstance, when combining with the phase shifted geometric wave, the integral above is indeed an intersecting Lagrangian distribution introduced by Melrose and Uhlmann \cite{melrose1979lagrangian}. The intersecting Lagrangian structure near the geometric wavefront was discussed in \cite{ford2018wave} for the Euclidean manifold with conic singularities. It is worth to point out that the diffractive propagator is the same whether we pass through such a phase shift; the assumptions are designed purely to cancel out singularities at the geometric wavefront. 
\end{remark}

A straightforward computation shows that 
$$\Upsilon_s\left(t,q_1,q_2\right)=\left[(\Phi^s_{T_x})^* \Upsilon_0\right]\left(t,q_1,q_2\right).$$
The diffractive propagator therefore can be written as 
\begin{equation}
\label{Diffrctive Propagator}
    E_{\text{D}}(t)=\int^{\infty}_0 \frac{-\sin\pi\alpha \cdot e^{i\alpha\Delta\theta}}{4\pi\sqrt{r_1(s)r_2(s)}}\left(\int_{\RR}e^{i\lambda(t-r_2(s)-r_1(s))}\tilde{a}(r_1(s),r_2(s),\lambda)d\lambda\right)(e^{-i\theta_1(s)}+e^{i\theta_2(s)})ds,
\end{equation}
where $r_1(s)=\sqrt{(x_1+s)^2+y_1^2}$, $r_2(s)=\sqrt{(x_2+s)^2+y_2^2}$,
 $\theta_1(s)=\arctan(\frac{y_1}{x_1+s})$, $\theta_2(s)=\arctan(\frac{y_2}{x_2+s})$ and $\Delta\theta=(\theta_1(s)-\theta_2(s))-(\theta_1-\theta_2)$. Use the Fourier transform of the Heaviside function and make the change of variable $\rho=\lambda\mu$:   
\begin{equation*}
    \begin{split}
        & \int^{\infty}_0 \frac{ e^{i\alpha\Delta\theta}}{\sqrt{r_1(s)r_2(s)}}\left(\int_{\RR}e^{i\lambda(t-r_2(s)-r_1(s))}\tilde{a}(r_1(s),r_2(s),\lambda)d\lambda\right)(e^{-i\theta_1(s)}+e^{i\theta_2(s)})ds\\
        =& \int_{\RR} H(s)\frac{ e^{i\alpha\Delta\theta}}{\sqrt{r_1(s)r_2(s)}}\left(\int_{\RR}e^{i\lambda(t-r_2(s)-r_1(s))}\tilde{a}(r_1(s),r_2(s),\lambda)d\lambda\right)(e^{-i\theta_1(s)}+e^{i\theta_2(s)})ds\\
        =& \int_{\RR}\int_{\RR}e^{is\rho}\frac{1}{\rho+i0}\frac{ e^{i\alpha\Delta\theta}}{\sqrt{r_1(s)r_2(s)}}\left(\int_{\RR}e^{i\lambda(t-r_2(s)-r_1(s))}\tilde{a}(r_1(s),r_2(s),\lambda)d\lambda\right)(e^{-i\theta_1(s)}+e^{i\theta_2(s)})d\rho ds\\
        =& \int_{\RR}\int_{\RR}e^{is\lambda\mu}\frac{1}{\mu+i0}\frac{ e^{i\alpha\Delta\theta}}{\sqrt{r_1(s)r_2(s)}}\left(\int_{\RR}e^{i\lambda(t-r_2(s)-r_1(s))}\tilde{a}(r_1(s),r_2(s),\lambda)d\lambda\right)(e^{-i\theta_1(s)}+e^{i\theta_2(s)})d\mu ds\\
        =& \int_{\RR}\int_{\RR}\int_{\RR}e^{i\lambda (s\mu+(t-r_2(s)-r_1(s))}\frac{1}{\mu+i0}\frac{ e^{-i\alpha\Delta\theta}}{\sqrt{r_1(s)r_2(s)}}\tilde{a}(r_1(s),r_2(s),\lambda)(e^{-i\theta_1(s)}+e^{i\theta_2(s)})d\lambda d\mu ds. 
    \end{split}
\end{equation*}
Now we apply the stationary phase lemma in variables $(\mu,s)$ to the above integral, with the phase function $\phi=\mu s+(t-r_1(s)-r_2(s))$. The non-degenerate critical point is at $\mu=r_1'(0)+r_2'(0), s=0$. Thus its leading order singularity is given by
\begin{equation*}
    \begin{split}
       \frac{2\pi}{\sqrt{r_1r_2}}\int_{\RR}e^{i\lambda (t-r_2-r_1)}\frac{1}{r_1'(0)+r_2'(0)}(e^{-i\theta_1}+e^{i\theta_2})\lambda^{-1}d\lambda. 
    \end{split}
\end{equation*}
Therefore the diffraction coefficient, i.e., the principal symbol of the conormal distribution \eqref{Diffrctive Propagator}, is
\begin{equation*}
       -\frac{\sin\pi\alpha}{2\sqrt{r_1r_2}}\cdot\frac{e^{-i\theta_1}+e^{i\theta_2}}{\cos\theta_1+\cos\theta_2}\cdot\lambda^{-1}. 
\end{equation*}
Thus we summarize to the following theorem: 
\begin{theorem}
\label{theorem1}
For $|\theta_1-\theta_2|\neq\pi$, the diffractive Aharonov--Bohm propagator $E_{\emph{D}}(t)$ is a polyhomogenous conormal distribution with respect to $\{t=r_1+r_2\}$ with the oscillatory integral representation:
\begin{equation}
     \int_{\RR}e^{i\lambda(t-r_2-r_1)}a(r_1,r_2,\theta_1,\theta_2,\lambda)d\lambda.
\end{equation}
In particular, the principal symbol of $a$ is given by 
\begin{equation}
    a_0=-\frac{\sin\pi\alpha}{2\sqrt{r_1r_2}}\cdot\frac{e^{-i\theta_1}+e^{i\theta_2}}{\cos\theta_1+\cos\theta_2}\cdot\lambda^{-1},
\end{equation}
where $q_1=(r_1,\theta_1),q_2=(r_2,\theta_2)$ with $\alpha\in(0,1)$ depending on the magnetic flux at the solenoid. 
\end{theorem}
\begin{remark}
We can verify the principal symbol of the diffractive propagator formally using the diffractive coefficient computed in \cite{ford2017diffractive} and \cite{yang2020diffraction}. The diffraction coefficient of $e^{-it\sqrt{\pa}}$ away from $|\theta_1-\theta_2|=\pi$ is given by $-i(r_1r_2)^{-1/2}e^{-i\pi\nu}$ acting diagonally on each mode. Therefore, summing all the modes, the diffraction coefficient can be formally computed as 
\begin{equation*}
\begin{split}
     -i\sum_{k\in\mathbb{Z}}e^{-i\pi|k+\alpha|}e^{ik(\theta_1-\theta_2)}
     &= (-i) e^{-i\alpha\pi}\delta(\theta_1-\theta_2-\pi)-\sin(\pi\alpha)\frac{e^{-i\theta_1}+e^{i\theta_2}}{\cos\theta_1+\cos\theta_2}\\
     &\equiv -\sin(\pi\alpha)\frac{e^{-i\theta_1}+e^{i\theta_2}}{\cos\theta_1+\cos\theta_2}
\end{split}
\end{equation*}
modulo singularities at $|\theta_1-\theta_2|=\pi$, and the diffraction coefficient of $\frac{\sin(t\sqrt{\pa})}{\sqrt{\pa}}$ is therefore
\begin{equation*}
     \sum_{k\in\mathbb{Z}}\frac{-i}{2|\lambda|}(H(\lambda)e^{-i\pi|k+\alpha|}+H(-\lambda)e^{i\pi|k+\alpha|})e^{ik(\theta_1-\theta_2)}\equiv -\frac{1}{2\lambda}\cdot\sin(\pi\alpha)\cdot\frac{e^{-i\theta_1}+e^{i\theta_2}}{\cos\theta_1+\cos\theta_2}
\end{equation*}
modulo singularities at $|\theta_1-\theta_2|=\pi$, which agrees with the result in Theorem \ref{theorem1}.
\end{remark}

\begin{remark}
The principal amplitude of diffractive singularities depends on the magnetic flux $\alpha$. In particular, it vanishes as $\alpha\rightarrow\mathbb{Z}$. 
\end{remark}

\bibliography{reference}
\bibliographystyle{alpha}

\end{document}